\documentclass[12pt]{article}
\usepackage{graphicx} 
\usepackage[left=0.8in, right=0.8in, top=1.5in, bottom=1.5in]{geometry}

\usepackage[T1]{fontenc} 
\usepackage[utf8]{inputenc} 
\usepackage[english]{babel}
\usepackage{amssymb}
\usepackage{lipsum} 
\usepackage{lmodern}
\usepackage{amsmath}
\usepackage{amsthm}
\usepackage{bm}
\usepackage{mathtools}
\usepackage{braket}
\usepackage{appendix}
\usepackage{esint}
\newcommand{\abs}[1]{{\left|#1\right|}}
\newcommand{\norma}[1]{{\left\Vert#1\right\Vert}}

\newcommand{\dH}{d\mathcal{H}^{n-1}}

\usepackage{enumitem}
\usepackage{booktabs}
\usepackage{graphicx}
\usepackage{tikz}
\usetikzlibrary{patterns}
\usepackage{multicol}
\usepackage{caption}
\usepackage{bigints}
\usepackage[skins,theorems]{tcolorbox}
\tcbset{highlight math style={enhanced,
colframe=black,colback=white,arc=0pt,boxrule=1pt}}
\def\XXint#1#2#3{{\setbox0=\hbox{$#1{#2#3}{\int}$}
    \vcenter{\hbox{$#2#3$}}\kern-.5\wd0}}

\theoremstyle{definition}
\newtheorem{definizione}{Definition}[section]
\theoremstyle{plain}
\newtheorem{teorema}{Theorem}[section]

\newtheorem{lemma}[teorema]{Lemma}
\newtheorem{prop}[teorema]{Proposition}
\newtheorem{corollario}[teorema]{Corollary}
\theoremstyle{definition}
\newtheorem{esempio}{Example}[section]
\newtheorem{oss}[esempio]{Remark}

\newtheorem*{open*}{Open problems}
\DeclareMathOperator{\R}{\mathbb{R}}

\makeatletter
\newcommand{\myfootnote}[2]{\begingroup
	\def\@makefnmark{}%
	\addtocounter{footnote}{-1}%
	\footnote{\textbf{#1} #2}
	\endgroup}
\makeatother

\usepackage{hyperref}
\hypersetup{linktoc=none, bookmarksnumbered, colorlinks=true, linkcolor=cyan, citecolor=green}
\newlist{steps}{enumerate}{1}
\setlist[steps, 1]{label = Step \arabic*:}
\title{Talenti comparison results for solutions to $p$-Laplace equation on multiply connected domains}
\author{Luca Barbato, Francesco Salerno*}
\date{}

\newcommand{\Addresses}{{
\bigskip 
  
   \medskip

    \textit{E-mail address}, L.~Barbato: \texttt{l.barbato@ssmeridionale.it} 
  
   \medskip 
   
     \textit{E-mail address}, F.~Salerno* (corresponding author): \texttt{f.salerno@ssmeridionale.it} 
   \medskip

 \textsc{Mathematical and Physical Sciences for Advanced Materials and Technologies, Scuola Superiore Meridionale, Largo San Marcellino 10, 80138 Napoli, Italy.}

 \par\nopagebreak 

}} 

\begin{document}
\maketitle

\begin{abstract}

In the last years comparison results of Talenti type for Elliptic Problems have been widely investigated. In this paper we obtain a comparison result for solutions to the $p$-Laplace equation in multiply connected domains with Robin boundary condition on the exterior boundary and non-homogeneous Dirichlet boundary conditions on the interior one, generalizing the results obtained in \cite{ANT, AGM} to this type of domains. This will be a generalization to Robin boundary condition of the results obtained in \cite{B, B2}, with an improvement of the $L^2$ comparison in the case $p=2$. As a consequence, we obtain a Bossel-Daners and Saint-Venant type inequalities for multiply connected domains.  
\end{abstract}
\section{Introduction}
Let $\Omega \subset \mathbb{R}^n$ be an open bounded set. In his seminal paper \cite{T76}, Talenti investigated the Dirichlet problem
\begin{equation*}
    \begin{cases}
        -\Delta u = f & \text{in } \Omega,\\
        u = 0 & \text{on } \partial \Omega,
    \end{cases}
\end{equation*}
together with the associated symmetrized problem
\begin{equation*}
    \begin{cases}
        -\Delta v = f^\sharp & \text{in } \Omega^\sharp,\\
        v = 0 & \text{on } \partial \Omega^\sharp,
    \end{cases}
\end{equation*}
where $\Omega^\sharp$ denotes the ball centered at the origin with the same Lebesgue measure as $\Omega$, $f \in L^{\frac{2n}{n+2}}(\Omega)$ if $n>2$ (or $f \in L^p(\Omega)$ with $p>1$ if $n=2$), and $f^\sharp$ is the Schwarz rearrangement of $f$.  
The cornerstone of Talenti’s theory is the pointwise comparison
\begin{equation*}
    u^\sharp(x) \leq v(x) \quad \text{for all } x \in \Omega^\sharp,
\end{equation*}
which has had a profound impact on the development of symmetrization techniques in elliptic partial differential equations.

Since then, Talenti-type comparison principles have been extended in several directions. On the one hand, different boundary conditions have been considered, including Neumann conditions \cite{CNT} and Robin conditions \cite{ANT}. On the other hand, a wide variety of operators have been investigated: the $p$-Laplace operator with Robin boundary conditions \cite{AGM}, the Monge--Ampère operator in dimension two \cite{T81}, Hessian operators in arbitrary dimensions \cite{T89}, first-order Hamilton--Jacobi equations \cite{GN}, and anisotropic operators \cite{DPG,S}. More recently, significant attention has been devoted to the stability of Talenti-type inequalities; we refer to \cite{AL,AB,ABCMP,ABMP,Kim,MS} and the references therein.

The extension of these comparison results to \emph{multiply connected domains} presents substantial analytical difficulties and remains a challenging problem. While comparisons for the torsional rigidity associated with degenerate linear elliptic operators \cite{B} and for the $p$-Laplace operator \cite{AZ,B2} have been obtained, the presence of holes in the domain deeply affects the symmetrization procedure. Indeed, unlike the simply connected case—where level sets of the rearranged solution are naturally spherical—multiply connected geometries give rise to plateau regions, where the gradient vanishes. These regions prevent a direct application of classical Schwarz symmetrization techniques and require new ideas.

In this paper, we address these issues by considering a domain with a finite number of holes, endowed with Robin boundary conditions on the exterior boundary and non-homogeneous Dirichlet boundary conditions on the interior boundaries. Within this framework, we establish a Talenti-type inequality for the $p$-Laplace operator.  
The main novelty of our approach lies in the treatment of mixed, non-homogeneous boundary conditions through the introduction of suitably extended functions that are constant inside the holes. This strategy allows us to overcome the analytical difficulties caused by zero-gradient regions and to recover sharp comparison results in Lorentz spaces. The setting is the following.\\
Let $\Omega_0,\,\Omega_1,\dots,\,\Omega_m$ be bounded subsets of $\R^n$ such that 
\begin{itemize}
    \item[(i)] $\Omega_0$ is Lipschitz and connected,
    \item[(ii)] $\Omega_i\subset\subset\Omega_0$ and satisfies the exterior sphere condition $i=1,\dots,m$, in the sense that
    \begin{equation*}
        \forall\,x\in\partial\Omega_i\,\exists\,y\in\Omega_0\backslash\bigcup_{i=1}^m\overline{\Omega}_i,\,r>0\,:\,x\in\partial B_r(y),\,B_r(y)\subset\Omega_0\backslash\bigcup_{i=1}^m\overline{\Omega}_i,
    \end{equation*}
    \item[(iii)] $\overline{\Omega}_i\cap\overline{\Omega}_j=\emptyset$ if $i\neq j$ and $i,\,j\in\{1,\dots,m\}$.
\end{itemize}
 Set $\Omega=\Omega_0\backslash\bigcup_{i=1}^m\overline{\Omega}_i$. Let $\nu$ be the unit outward normal on the boundary of $\Omega$, let $f\in L^{p'}(\Omega_0)$ be a non-negative function, and let $\beta$ be a positive parameter. Consider the following problem
\begin{equation}\label{prob}
    \begin{cases}
        -\Delta_p u=f &\text{ in } \Omega\\
        \abs{\nabla u}^{p-2}\frac{\partial u}{\partial\nu}+\beta \abs{u}^{p-2}u=0 &\text{ on }\partial\Omega_0 \\
        u=c_i&\text{ on }\partial\Omega_i
    \end{cases},
\end{equation}
and the symmetrized one
\begin{equation}\label{probsymm}
    \begin{cases}
        -\Delta_p v=f^\sharp &\text{ in } A_\Omega=\Omega_0^\sharp\setminus S^\sharp\\
        \abs{\nabla v}^{p-2}\frac{\partial v}{\partial\nu}+\beta \abs{v}^{p-2}v=0 &\text{ on }\partial\Omega_0^\sharp \\
        v=\overline c &\text{ on }\partial S^\sharp
    \end{cases}
\end{equation}
where $S^\sharp=\left(\bigcup_{i=1}^{m}\Omega_i \right)^\sharp$ and the constants $c_i$, $\overline{c}$, $i=1,\dots,m$, are implicitly defined in Remark \ref{Def Ci}. Without loss of generality, we can assume that the constants $c_i$ are ordered.
The main results contained in this work are the following comparisons in the Lorentz space (see \S\ref{section2} for its precise definition).
\begin{teorema}\label{mainteor}
    Let $\Omega=\Omega_0\backslash\bigcup_{i=1}^m\overline{\Omega}_i$ where $\Omega_i$, $i=0,1,\dots,m$, satisfy conditions (i),(ii) and (iii) and let $u$, $v$ be the solutions to problems \eqref{prob} and \eqref{probsymm} respectively. Consider their constant extension $\tilde u$ and $\tilde v$ to $\Omega_0$ and $\Omega_0^\sharp$ respectively. Then we have
    \begin{equation*}
        \|\tilde u\|_{L^{k,1}(\Omega_0)}\leq\|\tilde v\|_{L^{k,1}(\Omega_0^\sharp)},\quad \forall\,0<k\leq \frac{n(p-1)}{(n-1)p}.
    \end{equation*}
    Moreover, if $p\geq n$ is an integer we have
    \begin{equation*}
        \|\tilde u\|_{L^{pk,p}(\Omega_0)}\leq\|\tilde v\|_{L^{pk,p}(\Omega_0^\sharp)},\quad \forall\,0<k\leq \frac{n(p-1)}{(n-1)p}.\
    \end{equation*}
\end{teorema}
\begin{corollario}\label{BM}
    In the same assumptions of Theorem \ref{mainteor}, if $p\geq n$ is an integer, then
    \begin{equation*}
        \|\tilde u\|_{L^1(\Omega_0)}\leq \|\tilde v\|_{L^1(\Omega_0^\sharp)}\quad \text{and}\quad \|\tilde u\|_{L^p(\Omega_0)}\leq \|\tilde v\|_{L^p(\Omega_0^\sharp)}.
    \end{equation*}
\end{corollario}
\begin{teorema}\label{mainteo2}
    Let $\Omega=\Omega_0\backslash\bigcup_{i=1}^m\overline{\Omega}_i$ where $\Omega_i$, $i=0,1,\dots,m$, satisfy conditions (i),(ii) and (iii).
    Suppose that $f\equiv1$ and let $u$ and $v$ be the solutions to \eqref{prob} and \eqref{probsymm} respectively.
    \begin{itemize}
        \item [(i)] If $1\leq p\leq \frac{n}{n-1}$, then
        \begin{equation*}
            \tilde u^\sharp(x)\leq \tilde v(x),\quad \forall x\in\Omega_0^\sharp,
        \end{equation*}
        \item [(ii)] if $p>\frac{n}{n-1}$ is an integer and $0<k\leq \frac{n(p-1)}{n(p-1)-p}$, then
        \begin{equation*}
            \|\tilde u\|_{L^{k,1}(\Omega_0)}\leq \|\tilde v\|_{L^{k,1}(\Omega_0^\sharp)}\quad \text{and}\quad \|\tilde u\|_{L^{pk,p}(\Omega_0)}\leq \|\tilde v\|_{L^{pk,p}(\Omega_0^\sharp)}.
        \end{equation*}
    \end{itemize}
\end{teorema}
As a consequence, in the case $f\equiv 1$, we have that
\begin{equation*}
    \|\tilde u\|_{L^1(\Omega_0)}\leq \|\tilde v\|_{L^1(\Omega_0^\sharp)}\quad \text{and}\quad \|\tilde u\|_{L^p(\Omega_0)}\leq \|\tilde v\|_{L^p(\Omega_0^\sharp)},
\end{equation*}
for $p>1$, while we have the pointwise comparison only for $p\leq\frac{n}{n-1}$.

As an application of our result, we obtain refined $L^p$ and $L^2$ comparisons in the linear case $p=2$, improving previously known results. Moreover, in the limit as the Robin parameter $\beta \to +\infty$, our estimates recover—and sharpen—the results obtained in \cite{B2}. In contrast with earlier works, which mainly focus on $L^1$ comparisons, our method yields inequalities in the full $L^p$ scale. As a consequence, we show that the annulus maximizes the Robin $p$-torsional rigidity among all multiply connected domains with prescribed measures of the exterior region and of the holes.

This study is also motivated by the growing interest in shape optimization problems for domains with holes. Recent contributions include eigenvalue problems with different boundary conditions \cite{DPP,PPT,PW2,PW}, quantitative inequalities \cite{AB2,CPP}, and two-phase problems \cite{F,MNP}. Our results fit naturally into this context, extending the findings of \cite{ANT,AGM} to multiply connected domains and generalizing the results of \cite{B,B2} to the Robin boundary setting.

The paper is organized as follows. In Section~\ref{section2} we introduce the notation and the main analytical tools. Section~\ref{section3} is devoted to the well-posedness of the problem in the class of multiply connected domains. In Section~\ref{section4} we prove the main comparison results, adapting and extending the techniques developed in \cite{ANT}. Finally, in Section~\ref{section5} we apply our results to derive a Bossel--Daners type inequality.
\section{Notions and preliminaries}\label{section2}
This section is devoted to the introduction of the main tools we will use in what follows. The first tool concerns rearrangements of functions, for which a comprehensive summary can be found in \cite{K}.\\
    \begin{definizione}
	Let $u: \Omega \to \R$ be a measurable function, the \emph{distribution function} of $u$ is the function $\mu : [0,+\infty[\, \to [0, +\infty[$ defined by
	$$
	\mu(t)= \abs{\Set{x \in \Omega \, :\,  \abs{u(x)} > t}},
	$$
    where here, and in the following, $\abs{\,\cdot\,}$ stands for the $n$-dimensional Lebesgue measure.
\end{definizione}

\begin{definizione}
	Let $u: \Omega \to \R$ be a measurable function and $\mu$ its distribution function, the \emph{decreasing rearrangement} of $u$, denoted by $u^\ast(\cdot)$, is defined as
    $$u^*(s)=\inf\{t\geq 0:\mu(t)<s\}.$$
	
	\noindent The \emph{Schwarz rearrangement} of $u$ is the function $u^\sharp $ whose level sets are balls with the same measure as the level sets of $u$. The functions $u^\sharp$ and $u^*$ are linked by the relation
	$$u^\sharp (x)= u^*(\omega_n \abs{x}^n)$$
    where $\omega_n$ denotes the measure of the $n$-dimensional unit ball.
\end{definizione}
\noindent It is easily checked that $u$, $u^*$, and $u^\sharp$ are equi-distributed, so it follows that
$$ \displaystyle{\norma{u}_{L^p(\Omega)}=\norma{u^*}_{L^p(0, \abs{\Omega})}=\lVert{u^\sharp}\rVert_{L^p(\Omega^\sharp)}}.$$
An important property of the decreasing rearrangement is the Hardy- Littlewood inequality, that is
\begin{equation*}
 \int_{\Omega} \abs{h(x)g(x)} \, dx \le \int_{0}^{\abs{\Omega}} h^*(s) g^*(s) \, ds,
\end{equation*}
whenever the product $hg$ is in $L^1(\Omega)$. So, by choosing $g=\chi_{\left\lbrace\abs{u}>t\right\rbrace}$, one has

\begin{equation*}
\int_{\abs{u}>t} \abs{h(x)} \, dx \le \int_{0}^{\mu(t)} h^*(s) \, ds.
\end{equation*}

\begin{definizione}
Let $0<p<+\infty$ and $0<q\le +\infty$. The Lorentz space $L^{p,q}(\Omega)$ is the space of those functions such that the quantity:
\begin{equation*}
    \norma{g}_{L^{p,q}} =
    \begin{cases}
   	\displaystyle{ p^{\frac{1}{q}} \left( \int_{0}^{\infty}  t^{q} \mu(t)^{\frac{q}{p}}\, \frac{dt}{t}\right)^{\frac{1}{q}}} & 0<q<\infty\\
	 \displaystyle{\sup_{t>0} \, (t^p \mu(t))} & q=\infty
	\end{cases}
\end{equation*}
is finite.
\end{definizione}
	
\noindent Let us observe that for $p=q$ the Lorentz space coincides with the  $L^p$ space, as a consequence of the well known \emph{Cavalieri's Principle}

$$\int_\Omega \abs{g}^p =p \int_0^{+\infty} t^{p-1} \mu(t) \, dt.$$
	
\noindent See \cite{T94} for more details on Lorentz space.\\
Moreover, if $\Omega$ is an open set, the following coarea formula applies. A reference for results relative to the sets of finite perimeter and the coarea formula is \cite{AFP}.
\begin{teorema}[Coarea formula]
	Let $\Omega \subset \mathbb{R}^n$ be an open set. Let $f\in W^{1,1}_{\text{loc}}(\Omega)$ and let $u:\Omega\to\R$ be a measurable function. Then,
	\begin{equation}
        	\label{coarea}
		{\displaystyle \int _{\Omega}u(x)|\nabla f(x)|dx=\int _{\mathbb {R} }dt\int_{\Omega\cap f^{-1}(t)}u(y)\, d\mathcal {H}^{n-1}(y)}.
	\end{equation}
\end{teorema}
Also using the Coarea formula \eqref{coarea}, it is possible to deduce an explicit expression for $\mu$ in terms of integrals of $u$
\begin{equation*}
    \mu(t)=\abs{\{u>t\}\cap \{|\nabla u |=0\}}+ \int_t^{+\infty} \left(\int_{u=s}\frac{1}{\abs{\nabla u}}\, d\, \mathcal{H}^{n-1}\right)\, ds,
\end{equation*}
as a consequence, for almost all $t\in (0,+\infty)$,
    \begin{equation}\label{brothers1}
        \infty>-\mu'(t)\geq\displaystyle \int_{u=t}\dfrac{1}{|\nabla u|}d\mathcal{H}^{n-1}
    \end{equation}
moreover if $\mu$ is absolutely continuous, equality holds in \eqref{brothers1}. We recall that the absolute continuity of $\mu$ is ensured by the following lemma (we refer, for instance, to \cite{BZ}).
\begin{lemma}\label{Broz}
Let $u\in W^{1,p}(\R^n)$, with $p\in(1,+\infty)$. The distribution function $\mu$ of $u$ is absolutely continuous if and only if 
\begin{equation*}
    \abs{\{0<u<  ||u^\sharp||_\infty\}\cap \{|\nabla u^\sharp| =0\}}=0.
\end{equation*}
\end{lemma}

Let us fix some notations. Let $u$ and $v$ be the solutions to \eqref{prob}, \eqref{probsymm}, respectively, and let $\tilde u$, $\tilde v$ be their constant extension. For $t\geq0$, we denote by
\begin{equation*}
    U_t=\{x\in\Omega\,:\,|\tilde u(x)|>t\},\quad \partial U_t^{int}=\partial U_t\cap \Omega,\quad \partial U_t^{ext}=\partial U_t\cap \partial \Omega,
\end{equation*}
and by
\begin{equation}\label{distribuzioemu}
    \mu(t)=\abs{U_t},\quad P_u(t)=Per(U_t),
\end{equation}
where $Per(\cdot)$ is the perimeter in the sense of De Giorgi. Using the same notation, we set
\begin{equation}\label{distribuzionephi}
    V_t=\{x\in\Omega^\sharp\,:\,|\tilde v(x)|>t\},\quad
    \phi(t)=\abs{V_t},\quad P_v(t)=Per(V_t).
\end{equation}

\section{Existence and properties of solutions}\label{section3}

In this section, we establish the well-posedness of the associated variational problem. Our aim is to minimize the following functional:
\begin{equation}\label{funzionale}
    \mathcal{F}(w)=\frac{1}{p}\int_{\Omega_0}\abs{\nabla w}^p\,dx+\frac{\beta}{p}\int_{\partial\Omega_0}\abs{w}^p\,\mathcal{H}^{n-1}-\int_{\Omega_0} fw\,dx,
\end{equation}
over the space 
\begin{equation}\label{Spazio}
    \mathcal{X}=\{u\in W^{1,p}(\Omega_0)\,: \text{ $\nabla u=0$ on $\Omega_i$ $\forall\,i=1,\dots,m$}\},
\end{equation}
where $\Omega_i$, $i=0,1,\dots,m$, satisfy conditions (i),(ii) and (iii). \\
This approach is motivated by the fact that, given the structure of the functional, a minimizer of $\mathcal{F}$ on $\mathcal{X}$ is a weak solution to \eqref{prob}. In order to apply the Direct Method, the coercivity and w.l.s.c. of the functional \eqref{funzionale} are classical results and can be inferred from standard arguments (such as the convexity of the gradient and boundary terms, which are for instance treated in \cite[\S 2]{AGM}), in our case the validity depends on the admissible space $\mathcal{X}$ being closed with respect to the weak topology.
\begin{prop}\label{chiusura}
    Let $\Omega_0\subset\R^n$ be connected. Then the class $\mathcal{X}$, defined in \eqref{Spazio}, is closed with respect to the weak topology in $W^{1,p}$.
\end{prop}
\begin{proof}
    Since $\mathcal{X}$ is a vector space, it  suffices to prove the closure with respect to the strong convergence in $W^{1,p}$. Let $\{u_j\}_{j\in\mathbb{N}}\subset \mathcal{X}$ be such that $u_j\xrightarrow{W^{1,p}(\Omega_0)}u$, i.e.,
    \begin{equation*}
        \|u_j-u\|_{L^p(\Omega_0)}\rightarrow 0, \quad \|\nabla u_j-\nabla u\|_{L^p(\Omega_0)}\rightarrow 0.
    \end{equation*}
    Note that $u\in W^{1,p}(\Omega_0)$ since it is a closed space. We have to prove that $\nabla u=0$ on $\Omega_i$ for all $i=1,\dots,m$. So, for a fixed $i=1,\dots,m$, since $u_j\in \mathcal{X}$ and $\Omega_i\subset \Omega_0$ we have that
    \begin{equation*}
        0\leftarrow \|\nabla u_j-\nabla u\|_{L^p(\Omega_0)}\geq \|\nabla u_j-\nabla u\|_{L^p(\Omega_i)}=\|\nabla u\|_{L^p(\Omega_i)}
    \end{equation*}
    and this concludes the proof.
\end{proof}

By combining the coercivity and weak lower semicontinuity with the weak closure of $\mathcal{X}$ (proven in Prop. \ref{chiusura}), we have the existence of at least one minimizer $u \in \mathcal{X}$ for the functional. Furthermore, the uniqueness of the minimizer (and thus of the weak solution to the PDE) is ensured, as the functional $\mathcal{F}$ is the sum of strictly convex terms (the $p$-powers of the gradient and the trace, for $p>1$) and a linear term.
\begin{oss}[Implicit definition of $c_i$]\label{Def Ci}
     Let $u\in C^{1,\alpha}(\overline{\Omega})$ be a solution to \eqref{prob} (for instance, see \cite[Theorem 2]{Lieb} for the regularity of such solution) and multiply by $v\in\mathcal{X}$ and integrate on $\Omega$, by divergence theorem we obtain
    \begin{equation*}
        \int_\Omega\abs{\nabla u}^{p-2}\nabla u\cdot\nabla v\,dx+\beta\int_{\partial\Omega_0}\abs{u}^{p-2}uv\,\dH-\sum_{i=1}^m\tilde c_i\int_{\partial\Omega_i}\abs{\nabla u}^{p-2}\frac{\partial u}{\partial\nu}\,\dH=\int_\Omega fv\,dx,
    \end{equation*}
    where $\tilde c_i$ are the constant values that $v$ assumes on $\partial\Omega_i$. Finally, we choose the solution $u$ to be such that
    \begin{equation}\label{compatibility}
        \int_{\partial\Omega_i}\abs{\nabla u}^{p-2}\frac{\partial u}{\partial\nu}\,\dH=\int_{\Omega_i}f\,dx\quad \forall\,i=1,\dots,m.
    \end{equation}
    so that we have the following variational characterization of the solution as
    \begin{equation}\label{eqdebole}
        \int_{\Omega_0}\abs{\nabla u}^{p-2}\nabla u\cdot\nabla v\,dx+\beta\int_{\partial\Omega_0}\abs{u}^{p-2}uv\,\dH=\int_{\Omega_0} fv\,dx\quad\forall\,v\in\mathcal{X}.
    \end{equation}
    In the same way, we implicitly define the constant $\overline{c}$.
\end{oss}
\begin{lemma}\label{segno ci}
    The constants $c_i$ such that \eqref{prob} has solution are positive.
\end{lemma}
\begin{proof}
    We remark that if we take the modulus of $w$ in \eqref{funzionale}, using $w\leq\abs{w}$, the functional decreases, therefore we can restrict our analysis to $w\geq0$. Also, if $u$ solves \eqref{prob} then $u$ is $p$-superharmonic, thus by maximum principle \cite{V} attains its minimum in $x_0\in\partial\Omega$. Arguing by contradiction, if it exists $1\leq i\leq m$ such that $x_0\in\partial\Omega_i$ by Hopf lemma \cite[Theorem 5.5.2]{PS} we have $\displaystyle{\frac{\partial u}{\partial\nu}(x_0)<0}$ , which contradicts \eqref{compatibility}, then $x_0\in\partial\Omega_0$. Moreover, the latter implies that $c_i>0$. 
\end{proof}
\section{Proof of the main results}\label{section4}
We start this section by proving some useful lemmas. The first result is a dedicated version of the well-known \emph{Gronwall Lemma} (see \cite{G} for the original version of this result) which we prove to make the work self-contained.
\begin{lemma}[Gronwall]\label{gronwall lemma}
    Let $\xi:[\tau_0,+\infty)\rightarrow\R$ be a continuous and differentiable function satisfying, for some nonnegative constant $C$, the following differential inequality
    \begin{equation}\label{hp Gronwall}
        \tau\xi'(\tau)\leq(q-1)\xi(\tau)+C
    \end{equation}
    for all $\tau\geq\tau_0$. Then
    \begin{equation*}
        \begin{split}
            \xi(\tau)&\leq \left(\xi(\tau_0)+\frac{C}{q-1} \right)\left(\frac{\tau}{\tau_0} \right)^{q-1}-\frac{C}{q-1},\\
            \xi'(\tau)&\leq \left(\frac{(q-1)\xi(\tau_0)+C}{\tau_0} \right)\left(\frac{\tau}{\tau_0} \right)^{q-2},
        \end{split}
    \end{equation*}
    for all $\tau\geq\tau_0$.
\end{lemma}
\begin{proof}
    Dividing by $\tau^{q}$ in \eqref{hp Gronwall} we obtain
    \begin{equation*}
        \frac{\xi'(\tau)}{\tau^{q-1}}-(q-1)\frac{\xi(\tau)}{\tau^{q}}=\left(\frac{\xi(\tau)}{\tau^{q-1}}\right)' \leq \frac{C}{\tau^{q}},
    \end{equation*}
    and, integrating from $\tau_0$ to $\tau$,
    \begin{equation*}
        \xi(\tau)\leq \left(\xi(\tau_0)+\frac{C}{q-1} \right)\left(\frac{\tau}{\tau_0} \right)^{q-1}-\frac{C}{q-1},
    \end{equation*}
    which proves the first statement. Using the hypothesis and the inequality just proved, we find
    \begin{equation*}
        \begin{split}
            \tau\xi'(\tau)&\leq(q-1)\xi(\tau)+C\leq(q-1)\left[\left(\xi(\tau_0)+\frac{C}{q-1} \right)\left(\frac{\tau}{\tau_0} \right)^{q-1}-\frac{C}{q-1}\right]+C\\
            &=(q-1)\left(\xi(\tau_0)+\frac{C}{q-1} \right)\left(\frac{\tau}{\tau_0} \right)^{q-1}=(\xi(\tau_0)(q-1)+C)\left(\frac{\tau}{\tau_0} \right)^{q-1}.
        \end{split}
    \end{equation*}
    Then, dividing both members by $\tau$, we find
    \begin{equation*}
        \xi'(\tau)\leq \left(\frac{(q-1)\xi(\tau_0)+C}{\tau_0} \right)\left(\frac{\tau}{\tau_0} \right)^{q-2}.
    \end{equation*}
\end{proof}
\begin{oss}\label{brotherziemer}
    From the symmetry of the problem \eqref{probsymm} follows that $v$ is a symmetric radially decreasing function. In particular, this implies that the maximum is attained on $\partial S^\sharp$. Then the only critical points are those on $\partial S^\sharp$ and,
    by Lemma \ref{Broz}, the distribution function of $v$ is absolutely continuous. 
\end{oss}
\begin{lemma}
    Let $u$ and $v$ be the solutions to \eqref{prob}-\eqref{probsymm}, respectively. Let $\mu$ and $\phi$ be defined as in \eqref{distribuzioemu}-\eqref{distribuzionephi}. Then for a.e. $t>0$ we have
    \begin{equation}\label{compsymm}
        \gamma_n\phi(t)^{\left(1-\frac{1}{n}\right)\frac{p}{p-1}}=\left(\int_0^{\phi(t)}f^\ast(s)\,ds\right)^\frac{1}{p-1}\left(- \phi'(t)+\frac{1}{\beta^\frac{1}{p-1}}\int_{\partial V_t\cap\Omega_0^\sharp}\frac{1}{\tilde v}\,d\mathcal{H}^{n-1}\right),
    \end{equation}
    while for almost all $t>0$ it holds
    \begin{equation}\label{comp}
        \gamma_n\mu(t)^{\left(1-\frac{1}{n}\right)\frac{p}{p-1}}\leq\left(\int_0^{\mu(t)}f^\ast(s)\,ds\right)^\frac{1}{p-1}\left(- \mu'(t)+\frac{1}{\beta^\frac{1}{p-1}}\int_{\partial U^{ext}_t}\frac{1}{\tilde u}\,d\mathcal{H}^{n-1}\right).
    \end{equation}
\end{lemma}
\begin{proof}
    Let $t>0$ and $t\neq c_i$ for all $i=1,\dots, m$ and $h>0$ be sufficiently small such that if $c_i<t<c_{i+1}$ then $c_i<t+h<c_{i+1}$ (for such choice of $t$, the corresponding super-level set does not intersect the holes, where the gradient of the constant extension vanishes). Define
    \begin{equation*}
        \varphi_h(x)=\begin{cases}
                0&\text{ if } 0\leq \tilde u<t \\
                \tilde u-t&\text{ if } t\leq \tilde u<t+h\\
                h &\text{ if }\tilde u\geq t+h
                    \end{cases}.
    \end{equation*}
    Using $\varphi_h$ as a test in the formulation \eqref{eqdebole}
    \begin{equation}\label{weakphi}
        \int_{\Omega_0}\abs{\nabla \tilde u}^{p-2}\nabla \tilde u\cdot\nabla\varphi_h\,dx+\beta\int_{\partial\Omega_0}\tilde u^{p-1}\varphi_h\,\dH=\int_{\Omega_0}f\varphi_h\,dx,
    \end{equation}
    by the definition of $\varphi_h$, we get
    \begin{equation*}
        \begin{split}
            \int_{\Omega_0}\abs{\nabla \tilde u}^{p-2}\nabla \tilde u\cdot\nabla\varphi\,dx=& \int_{U_t\setminus U_{t+h}}\abs{\nabla u}^p\,dx,\\
            \int_{\partial\Omega_0}\tilde u^{p-1}\varphi\,\dH=&h\int_{\partial U_{t+h}^{ext}}\tilde u^{p-1}\,\dH+\int_{\partial U_{t}^{ext}\setminus\partial U_{t+h}^{ext}}\tilde u^{p-1}(\tilde u-t)\,\dH,\\ \int_{\Omega_0}f\varphi\,dx=&h\int_{U_{t+h}}f\,dx+\int_{U_t\setminus U_{t+h}}f(\tilde u-t)\,dx
        \end{split}
    \end{equation*}
    thus dividing \eqref{weakphi} by $h$ becomes
    \begin{equation}\label{INVENTO}
    \begin{split}
        \frac{1}{h}\int_{U_t\setminus U_{t+h}}\abs{\nabla \tilde u}^p\,dx+\frac{\beta}{h}\int_{\partial U_{t}^{ext}\setminus\partial U_{t+h}^{ext}}\tilde u^{p-1}(\tilde u-t)\,\dH&+\beta\int_{\partial U_{t+h}^{ext}}\tilde u^{p-1}\,\dH\\
        &=\int_{U_{t+h}}f\,dx+\frac{1}{h}\int_{U_t\setminus U_{t+h}}f(\tilde u-t)\,dx.
    \end{split}
    \end{equation}
    We remark that
    \begin{equation*}
        \frac{\beta}{h}\int_{\partial U_{t}^{ext}\setminus\partial U_{t+h}^{ext}}\tilde u^{p-1}(\tilde u-t)\,\dH<\beta\int_{\partial U_{t}^{ext}\setminus\partial U_{t+h}^{ext}}\tilde u^{p-1}\,\dH,
    \end{equation*}
    and the last, as $h\rightarrow 0$, it becomes the integral of $\tilde u^{p-1}$ on a set whose measure goes to zero. Then, letting $h\rightarrow0$ in \eqref{INVENTO}, we have
    \begin{equation*}
        \int_{U_t}f\,dx=\int_{\partial U_t}g(x)\,\dH,
    \end{equation*}
    where 
    \begin{equation*}
        g(x)=\begin{cases}
            \abs{\nabla \tilde u}^{p-1}&\text{ on }\partial U_t^{int} \\
            \beta\tilde u^{p-1}&\text{ on }\partial U_t^{ext}
        \end{cases}.
    \end{equation*}
    We remark that, from the assumptions on $t$ and $t+h$, the function $g(x)$ is always non-zero.
    By Cauchy-Schwarz inequality and the previous identity
    \begin{equation*}
        \begin{split}
            Per(U_t)&=\left(\int_{\partial U_t}g(x)^\frac{1}{p}\frac{1}{g(x)^\frac{1}{p}}\,\dH\right)\leq\left(\int_{\partial U_t}g(x)\,\dH\right)^\frac{1}{p}\left(\int_{\partial U_t}\frac{1}{g(x)^\frac{1}{p-1}}\,\dH\right)^{1-\frac{1}{p}}\\ &\leq\left(\int_0^{\mu(t)}f^\ast(s)\,ds\right)^\frac{1}{p}\left(-\mu'(t)+\frac{1}{\beta}\int_{\partial U_t^{ext}}\frac{1}{u}\,\dH\right),
        \end{split}
    \end{equation*}
    on the other hand, by the isoperimetric inequality, we find
    \begin{equation*}
        n\omega_n^\frac{1}{n}\mu(t)^{1-\frac{1}{n}}\leq\left(\int_0^{\mu(t)}f^\ast(s)\,ds\right)^\frac{1}{p}\left(-\mu'(t)+\frac{1}{\beta}\int_{\partial U_t^{ext}}\frac{1}{u}\,\dH\right).
    \end{equation*}
     Lastly, if $v$ solves \eqref{probsymm}, then \eqref{brothers1} holds with equality sign thanks to
    remark \ref{brotherziemer}. In particular, this implies that for $v$ the last inequality holds as an equality.
\end{proof}
As already observed the solutions $u$ and $v$ to \eqref{prob} and \eqref{probsymm}, respectively, attain their minima on the exterior boundary of their respective domain. Moreover, by the equations \eqref{prob}-\eqref{probsymm}, the equimeasurability property of the rearrangement and the isoperimetric inequality, we find
\begin{equation*}
      v_mPer(\Omega_0^\sharp)=\int_{\partial \Omega_0^\sharp}\tilde v\,\dH=\int_{\Omega_0^\sharp} f^\sharp\,dx=\int_{\partial\Omega_0}\tilde u\,\dH\geq u_mPer(\Omega_0)\geq u_mPer(\Omega_0^\sharp),
\end{equation*}
where
\begin{equation*}
u_m=\min_{\overline\Omega} u,\quad v_m=\min_{\overline\Omega^\sharp} v.
\end{equation*}
A consequence of the last inequality is that
\begin{equation*}
    \mu(t)\leq\phi(t)=\abs{\Omega_0},\quad \forall\, t\leq v_m
\end{equation*}
The next lemma deals boundary term on $\partial\Omega_0$ in the previous result.
\begin{lemma}\label{lemma3.5}
    For all $t\geq v_m$ we have
    \begin{equation*}
       \int_0^t\tau^{p-1}\left(\int_{\partial V_\tau\cap\partial\Omega_0^\sharp}\frac{1}{\tilde v}\,\dH\right)\,d\tau=\frac{1}{p\beta}\int_0^{\abs{\Omega}}f^*(s)\,ds,
    \end{equation*}
    and
    \begin{equation*}
      \int_0^t\tau^{p-1}\left(\int_{\partial U_\tau^{ext}}\frac{1}{\tilde u}\,\dH\right)\,d\tau\leq\frac{1}{p\beta}\int_0^{\abs{\Omega}}f^*(s)\,ds.
    \end{equation*}
\end{lemma} 
\begin{proof}
    Let us consider the quantity
    \[\tau^{p-1}\left(\int_{\partial U_\tau^{ext}}\frac{1}{\tilde u}\,\dH\right),\]
    and integrate in $(0,+\infty)$ by Fubini's theorem and the equation \eqref{prob}
    \begin{equation*}
        \begin{aligned}
            \int_0^{+\infty}\tau^{p-1}\left(\int_{\partial U_\tau^{ext}}\frac{1}{\tilde u}\,\dH\right)\,d\tau=\int_{\partial\Omega_0}\left(\int_0^{\tilde u(x)}\frac{\tau^{p-1}}{\tilde u(x)}\,d\tau\right)\,\dH\\=\frac{1}{p}\int_{\partial\Omega_0}\tilde u^{p-1}\,\dH=-\frac{1}{p\beta}\int_{\Omega_0}\mathrm{div}(\abs{\nabla u}^{p-2}\nabla u)\,dx\\=-\frac{1}{p\beta}\int_{\Omega}\mathrm{div}(\abs{\nabla u}^{p-2}\nabla u)\,dx=\frac{1}{p\beta}\int_0^\abs{\Omega}f^\ast(s)\,ds,
        \end{aligned}
    \end{equation*}
    and the same equality holds with $v$.\\
    Given that $\tilde u$ is positive we have for all $t\geq0$ by monotonicity of the integral
    \begin{equation*}
      \int_0^t\tau^{p-1}\left(\int_{\partial U_\tau^{ext}}\frac{1}{\tilde u}\,\dH\right)\,d\tau\leq\frac{1}{p\beta}\int_0^{\abs{\Omega}}f^*(s)\,ds,
    \end{equation*}
    on the other hand, by symmetry, $\partial V_t\cap\partial\Omega_0^\sharp$ is empty whenever $t\geq v_m$ which implies the equality with $\tilde v$.
\end{proof}
\begin{lemma}\label{LemmaPositivis}
    Let $g:\R\to\R$ be a positive increasing function. Let $n>1$ be an integer and $p\geq n$. Then the function
    \begin{equation*}
        h(l)=l^{-\left(1-\frac{1}{n}\right)\frac{p}{p-1}}\int_0^lg(s)\,ds,
    \end{equation*}
    is non-decreasing.
\end{lemma}

\begin{proof}
For simplicity set
\begin{equation*}
G(l)=\int_{0}^l g(s)\,ds, \qquad
\alpha=\left(1-\frac{1}{n}\right)\frac{p}{p-1}.
\end{equation*}
Then we compute the first derivative of $h$ as
\begin{equation*}
h'(l)
= l^{-\alpha-1}\left(l g(l)-\alpha G(l)\right).
\end{equation*}
Since $g$ is increasing, we have $g(s)\le g(l)$ for every $s\leq l$, hence
\begin{equation*}
G(l)=\int_0^l g(s)\,ds \le lg(l).
\end{equation*}
Moreover, the condition $p\geq n$ yields
\begin{equation*}
\alpha=\left(1-\frac1n\right)\frac{p}{p-1}\leq 1.
\end{equation*}
Combining these two inequalities gives
\begin{equation*}
\alpha G(l)\leq lg(l),
\end{equation*}
so the bracket in the expression for $h'$ is nonnegative. Therefore $h'(l)\geq0$ for all $l>0$, and the function $h$ is non-decreasing.
\end{proof}

Now we are able to prove Theorem \ref{mainteor} and Theorem \ref{mainteo2}.
\begin{proof}[Proof of Theorem \ref{mainteor}]
    Let $0<k\leq \frac{n(p-1)}{p(n-1)}$, so the quantity $\delta=\frac{1}{k}-\frac{(n-1)p}{n(p-1)}$ is positive. We start multiplying by $t^{p-1}\mu(t)^\delta$ in \eqref{comp} and we integrate on $[0,\tau]$ with $\tau\geq v_m$. Then, by Lemma \ref{lemma3.5}, we obtain
    \begin{equation}\label{16}
    \begin{split}
        \int_0^\tau \gamma_n t^{p-1}\mu(t)^\frac{1}{k}\,dt&\leq \int_0^\tau (-\mu'(t))t^{p-1}\mu(t)^\delta\left(\int_0^{\mu(t)}f^*(s)\,ds \right)^\frac{1}{p-1}\,dt\\&+\frac{\abs{\Omega_0}^\delta}{p\beta^\frac{p}{p-1}}\left(\int_0^\abs{\Omega_0}f^*(s)\,ds \right)^\frac{1}{p-1}\int_0^{\abs{\Omega}}f^*(s)\,ds.
   \end{split}
    \end{equation}
    If we define
    \begin{equation*}
        F_1(l)=\int_0^l \omega^\delta\left(\int_0^\omega f^*(s)\,ds\right)^\frac{1}{p-1}\,d\omega,
    \end{equation*}
    then the first term on the right-hand side in \eqref{16} is
    \begin{equation*}
        \int_0^\tau (-\mu'(t))t^{p-1}\mu(t)^\delta\left(\int_0^{\mu(t)}f^*(s)\,ds \right)^\frac{1}{p-1}\,dt=-\int_0^\tau t^{p-1}\frac{d}{dt}F_1(\mu(t))\,dt.
    \end{equation*}
    We can now integrate by parts both sides of \eqref{16} and, rearranging terms we find
    \begin{equation*}
        \begin{split}
            \tau^{p-1}\left(\int_0^\tau \gamma_n \mu(t)^\frac{1}{k}\,dt+F_1(\mu(\tau))\right)&\leq (p-1)\int_0^\tau t^{p-2}\left(\int_0^t \gamma_n \mu(s)^\frac{1}{k}\,ds+F_1(\mu(\tau))\right)\,dt\\
            &+\frac{\abs{\Omega_0}^\delta}{p\beta^\frac{p}{p-1}}\left(\int_0^\abs{\Omega_0}f^*(s)\,ds \right)^\frac{1}{p-1}\int_0^{\abs{\Omega}}f^*(s)\,ds.
        \end{split}
    \end{equation*}
    If we set now
    \begin{equation*}
    \begin{split}
        \xi_\mu(\tau)=&\int_0^\tau t^{p-2}\left(\int_0^t \gamma_n \mu(s)^\frac{1}{k}\,ds+F_1(\mu(\tau))\right)\,dt,\\ C_1=&\frac{\abs{\Omega_0}^\delta}{p\beta^\frac{p}{p-1}}\left(\int_0^\abs{\Omega_0}f^*(s)\,ds \right)^\frac{1}{p-1}\int_0^{\abs{\Omega}}f^*(s)\,ds,
   \end{split}
    \end{equation*}
    we can apply the Gronwall Lemma with $\tau_0=v_m$ and $q=p$, getting
    \begin{equation*}
        \tau^{p-2}\left(\int_0^\tau \gamma_n\mu(s)^\frac{1}{k}\,ds+F_1(\mu(\tau)) \right)\leq \left(\frac{(p-1)\xi_\mu(v_m)+C_1}{v_m}\right)\left(\frac{\tau}{v_m}\right)^{p-2},
    \end{equation*}
    where 
    \begin{equation}\label{xi}
        \xi_\mu(v_m)=\int_0^{v_m}t^{p-2}\left(\int_0^t\gamma_n\mu(s)^\frac{1}{k}\,ds+F_1(\mu(t))\right)\,dt.
    \end{equation}
    We remark that the previous inequality becomes an equality if we replace $\mu(t)$ with $\phi(t)$ from \eqref{compsymm}. Since $\mu(t)\leq\phi(t)$ for all $t\leq v_m$ and $F_1(\cdot)$ is monotone increasing, we have that $F_1(\mu(t))\leq F_1(\phi(t))$ for all $t\leq v_m$. This implies that $\xi_\mu(v_m)\leq\xi_\phi(v_m)$ and, from \eqref{xi} follows that 
    \begin{equation*}
        \int_0^\tau \gamma_n\mu(s)^\frac{1}{k}\,ds+F_1(\mu(\tau))\leq \int_0^\tau \gamma_n\phi(s)^\frac{1}{k}\,ds+F_1(\phi(\tau)).
    \end{equation*}
    Finally passing to the limit as $\tau\rightarrow+\infty$, we get
    \begin{equation*}
        \int_0^{+\infty}\mu(t)^\frac{1}{k}\,dt\leq \int_0^{+\infty}\phi(t)^\frac{1}{k}\,dt,
    \end{equation*}
    that is
    \begin{equation*}
        \|\tilde u\|_{L^{k,1}(\Omega_0)}\leq \|\tilde v\|_{L^{k,1}(\Omega_0^\sharp)}, \quad\forall\,0<k\leq\frac{n(p-1)}{p(n-1)}.
    \end{equation*}
    To prove the second part of the statement, is enough to show that
    \begin{equation*}
        \int_0^\infty t^{p-1}\mu(t)^\frac{1}{k}\,dt\leq\int_0^\infty t^{p-1}\phi(t)^\frac{1}{k}\,dt,
    \end{equation*}
    where $0<k\leq\frac{n(p-1)}{p(n-1)}$. Therefore, let us consider again \eqref{16} if we integrate by parts the first term on the right-hand side and let $\tau\to\infty$ we obtain
    \begin{equation*}
        \int_0^\infty\gamma_nt^{p-1}\mu(t)^\frac{1}{k}\,dt\leq(p-1)\int_0^\infty t^{p-2}F_1(\mu(t))\,dt+C_1,
    \end{equation*}
    and the same result (with equality sign) if we perform the same computation with $\phi$. Hence, it is enough to show that
    \begin{equation*}
        \int_0^\infty t^{p-2}F_1(\mu(t))\,dt\leq\int_0^\infty t^{p-2}F_1(\phi(t))\,dt.
    \end{equation*}
    To this aim, we multiply \eqref{16} by $t^{p-1}F_1(\mu(t))\mu(t)^{-\frac{(n-1)p}{n(p-1)}}$ and integrate. By lemma \ref{LemmaPositivis} the function $G_2(l)=F_1(l)l^{-\frac{(n-1)p}{n(p-1)}}$ is non decreasing and we obtain
    \begin{equation*}
        \int_0^\tau\gamma_n t^{p-1}F_1(\mu(t))\,dt\leq\int_0^\tau t^{p-1}\left(-\frac{d}{dt}F_2(\mu(t))\right)\,dt+C_2,
    \end{equation*}
    where
    \begin{equation*}
        \begin{split}
            F_2(l)&=\int_0^lG_2(\omega)\left(\int_0^\omega f^*(s)\,ds\right)^\frac{1}{p-1}d\omega\\
            C_2&=F_1(\abs{\Omega_0})\abs{\Omega_0}^{-\frac{(n-1)p}{n(p-1)}}\left(\int_0^{\abs{\Omega_0}}f^*(s)\,ds\right)^\frac{1}{p-1}\frac{\int_0^{\abs{\Omega}} f^*(s)\,ds}{p\beta^\frac{p}{p-1}}.
        \end{split}
    \end{equation*}
    Now we can integrate by parts and rearrange the term to get
    \begin{equation*}
        \begin{split}
            \tau&\int_0^\tau\gamma_nt^{p-2}F_1(\mu(t))\,dt+\tau\int_0^\tau t^{p-2}\left(\frac{dF_2(\mu(t))}{dt}\right)\leq\\&\int_0^\tau\int_0^t\gamma_n s^{p-2}F_1(\mu(s))\,ds\,dt+\int_0^\tau\int_0^t s^{p-2}\left(\frac{dF_2(\mu(s))}{ds}\right)\,ds\,dt+C_2,
        \end{split}
    \end{equation*}
    that is lemma \ref{gronwall lemma} with $q=2$, which implies
    \begin{equation*}
        \begin{split}
            \tau&\int_0^\tau\gamma_nt^{p-2}F_1(\mu(t))\,dt+\tau\int_0^\tau t^{p-2}\left(\frac{dF_2(\mu(t))}{dt}\right)\leq\\ &\frac{1}{v_m}\left(\int_0^{v_m}\int_0^t\gamma_ns^{p-2}F_1(\mu(s))\,ds\,dt+\int_0^{v_m}\int_0^ts^{p-2}\left(\frac{dF_2(\mu(s))}{ds}\right)\,ds\,dt+C_2\right),
        \end{split}
    \end{equation*}
    that are satisfied with equality if we make the computation with $\phi$ instead of $\mu$, and again the right-hand side are ordered because $F_2(\mu(t))$ is decreasing. Hence we get
    \begin{equation}\label{779}
        \begin{split}
            \gamma_n\int_0^\tau t^{p-2}F_1(\phi(t))-t^{p-2}F_1(\mu(t))\,dt&\geq\tau^{p-2}(F_2(\phi(\tau)-F_2(\mu(\tau)))\\+&(p-2)\int_0^\tau t^{p-3}(F_2(\phi(t)-F_2(\mu(t)))\,dt.
        \end{split}
    \end{equation}
    
    Now notice that by monotonicity of $F_2$ we have
    \begin{equation}\label{vanish}
        \begin{split}
            \tau^{p-2} F_2(\mu(\tau))&\leq\tau^{p-2}\mu(\tau)G_2(\mu(\tau))\left(\int_0^{\mu(\tau)}f^*(s)\,ds\right)^\frac{1}{p-1}\\
            &=\left(\int_{u>\tau}\tau^{p-2}\,dx\right)\mu(\tau)G_2(\mu(\tau))\left(\int_0^{\mu(\tau)}f^*(s)\,ds\right)^\frac{1}{p-1}\\
            &\leq\left(\int_{u>\tau}u^{p-2}\,dx\right)\mu(\tau)G_2(\mu(\tau))\left(\int_0^{\mu(\tau)}f^*(s)\,ds\right)^\frac{1}{p-1},
        \end{split}
    \end{equation}     
    therefore, if we send $\tau\to+\infty$, we have that the term $\tau^{p-2}F_2$ goes to zero. Then, thanks to \eqref{vanish}, if $p=2$ the claim is proved by sending $\tau\rightarrow+\infty$ in \eqref{779}, otherwise we have to prove
    \begin{equation}\label{claimfin}
        \int_0^\infty t^{p-3}F_2(\phi(t))\,dt\geq\int_0^\infty t^{p-3}F_2(\mu(t))\,dt.
    \end{equation}
    To prove \eqref{claimfin} we proceed recursively until the power of $t$ becomes zero; at each step for $m=3,\dots,p$ we multiply \eqref{16} by $t^{p-1}F_{m-1}(\mu(t))\mu(t)^{-\frac{(n-1)p}{n(p-1)}}$ and observe that $G_m(l)=F_{m-1}(l)l^{-\frac{(n-1)p}{n(p-1)}}$ is increasing thanks to lemma \ref{LemmaPositivis} and we find
    \begin{equation*}
        \int_0^\tau\gamma_n t^{p-1}F_{m-1}(\mu(t))\,dt\leq\int_0^\tau t^{p-1}\left(-\frac{d}{dt}F_m(\mu(t))\right)\,dt+C_m,
    \end{equation*}
    where
    \begin{equation*}
        \begin{split}
            F_m(l)&=\int_0^lG_m(\omega)\left(\int_0^\omega f^*(s)\,ds\right)^\frac{1}{p-1}d\omega\\
            C_m&=F_{m-1}(\abs{\Omega_0})\abs{\Omega_0}^{-\frac{(n-1)p}{n(p-1)}}\left(\int_0^{\abs{\Omega_0}}f^*(s)\,ds\right)^\frac{1}{p-1}\frac{\int_0^{\abs{\Omega}} f^*(s)\,ds}{p\beta^\frac{p}{p-1}}.
        \end{split}
    \end{equation*}
    Now we can integrate by parts and rearrange the term to get
    \begin{equation}
        \begin{split}
            \tau^{m-1}&\int_0^\tau\gamma_nt^{p-m}F_{m-1}(\mu(t))\,dt+\tau^{m-1}\int_0^\tau t^{p-m}\left(\frac{dF_m(\mu(t))}{dt}\right)\leq\\&(m-1)\int_0^\tau t^{m-2}\int_0^t\gamma_n s^{p-m}F_{m-1}(\mu(s))\,ds\,dt+\\&(m-1)\int_0^\tau t^{m-2}\int_0^t s^{p-m}\left(\frac{dF_m(\mu(s))}{ds}\right)\,ds\,dt+C_m,
        \end{split}
    \end{equation}
    that is lemma \ref{gronwall lemma} with $q=m$, which implies
    \begin{equation*}
        \begin{split}
            \tau^{m-2}&\int_0^\tau\gamma_nt^{p-m}F_{m-1}(\mu(t))\,dt+\tau^{m-2}\int_0^\tau t^{p-m}\left(\frac{dF_m(\mu(t))}{dt}\right)\leq\\ &\frac{1}{v_m}\left(\int_0^{v_m}t^{m-2}\int_0^t\gamma_ns^{p-m}F_{m-1}(\mu(s))\,ds\,dt+\int_0^{v_m}t^{m-2}\int_0^ts^{p-m}\left(\frac{dF_m(\mu(s))}{ds}\right)\,ds\,dt+C_m\right)\left(\frac{\tau}{v_m}\right)^{m-2},
        \end{split}
    \end{equation*}
    that are satisfied with equality if we make the computation with $\phi$ instead of $\mu$, and again the right-hand side are ordered because $F_m(\mu(t))$ is decreasing. Hence we get
    \begin{equation*}
        \begin{split}
            \gamma_n\int_0^\tau t^{p-m}F_{m-1}(\phi(t))-t^{p-m}F_{m-1}(\mu(t))\,dt&\geq\tau^{p-m}(F_m(\phi(\tau))-F_m(\mu(\tau)))\\+&(p-m)\int_0^\tau t^{p-m-1}(F_m(\phi(t)-F_m(\mu(t)))\,dt,
        \end{split}
    \end{equation*}
    that for $m=p$ concludes the proof by letting $\tau\to\infty$ (proceeding as in \eqref{vanish}).
\end{proof}

\begin{proof}[Proof of Theorem \ref{mainteo2}]
    Since $f=f^*\equiv1$, \eqref{comp} reads as
    \begin{equation}\label{20}
        \gamma_n\mu(t)^{\left(1-\frac{1}{n}-\frac{1}{p}\right)\frac{p}{p-1}}\leq -\mu'(t)+\frac{1}{\beta^\frac{1}{p-1}}\int_{\partial U_t^{ext}}\frac{1}{\tilde u} \dH.
    \end{equation}
    Now we multiply both sides of the last inequality by $t^{p-1}\mu(t)^\delta$, where $\delta=-\left(1-\frac{1}{n}-\frac{1}{p}\right)\frac{p}{p-1}$. We remark that, for $p\leq\frac{n}{n-1}$, $\delta\geq0$. Then, if we integrate on $[0,\tau]$, with $\tau\geq v_m$, and if we apply Lemma \ref{lemma3.5}, we find
    \begin{equation*}
        \int_0^\tau\gamma_n t^{p-1}\,dt\leq\int_0^\tau t^{p-1}\mu(t)^\delta(-\mu'(t))\,dt+\frac{\abs{\Omega_0}^{\delta}\abs{\Omega}}{p\beta^\frac{p}{p-1}}.
    \end{equation*}
    From \eqref{compsymm} and the rigidity in Lemma \ref{lemma3.5}, the last inequality holds as equality if we replace $\mu(t)$ with $\phi(t)$. Then we have
    \begin{equation*}
        \int_0^\tau t^{p-1}\phi(t)^\delta(-\phi'(t))\,dt\leq \int_0^\tau t^{p-1}\mu(t)^\delta(-\mu'(t))\,dt
    \end{equation*}
    Then, integrating by parts, we find
    \begin{equation*}
        -\tau^{p-1}\frac{\phi(\tau)^{\delta+1}}{\delta+1}+(p-1)\int_0^\tau t^{p-2}\frac{\phi(t)^{\delta+1}}{\delta+1}\,dt\leq -\tau^{p-1}\frac{\mu(\tau)^{\delta+1}}{\delta+1}+(p-1)\int_0^\tau t^{p-2}\frac{\mu(t)^{\delta+1}}{\delta+1}\,dt.
    \end{equation*}
    Now, applying the Gronwall's Lemma with
    \begin{equation*}
        \xi(\tau)=\int_0^\tau s^{p-2}\left(\frac{\mu(s)^{\delta+1}-\phi(s)^{\delta+1}}{\delta+1}\right)\,ds,
    \end{equation*}
    we obtain
    \begin{equation*}
        \tau^{p-2}\left(\frac{\mu(\tau)^{\delta+1}-\phi(\tau)^{\delta+1}}{\delta+1}\right)\leq (p-1)\frac{\tau^{p-2}}{v_m^{p-2}}\int_0^{v_m}s^{p-2}\left(\frac{\mu(s)^{\delta+1}-\phi(s)^{\delta+1}}{\delta+1}\right)\,ds.
    \end{equation*}
    Since $\mu(t)\leq\phi(t)$ for all $t\leq v_m$, we have that the integrand on the right-hand side is non-positive; the last inequality implies that
    \begin{equation*}
        \tau^{p-2}\left(\frac{\mu(\tau)^{\delta+1}-\phi(\tau)^{\delta+1}}{\delta+1}\right)\leq 0,
    \end{equation*}
    and then
    \begin{equation*}
        \mu(\tau)\leq\phi(\tau) \quad\forall\,\tau\geq v_m.
    \end{equation*}
    The comparison when $\tau\leq v_m$ was already known, therefore we get the pointwise comparison. \\
    
   We prove now the assertion $(ii)$, that is
    \begin{equation*}
        \|\tilde u\|_{L^{k,1}(\Omega_0)}\leq \|\tilde v\|_{L^{k,1}(\Omega^\sharp_0)}.
    \end{equation*}
    In order to achieve that, we need to show 
    \begin{equation*}
        \int_0^{+\infty}\mu(t)^\frac{1}{k}\,dt\leq\int_0^{+\infty}\phi(t)^\frac{1}{k}\,dt.
    \end{equation*}
    We multiply in \eqref{20} by $t^{p-1}\mu(t)^{\frac{1}{k}-\left(1-\frac{1}{n}-\frac{1}{p} \right)\frac{p}{p-1}}$ and integrate on $[0,\tau]$ with $\tau\geq v_m$. Then, from Lemma \ref{lemma3.5} and the fact that $\partial V_t\cap \partial\Omega_0^\sharp=\emptyset$ for all $t\geq v_m$ we obtain
    \begin{equation*}
        \int_0^\tau \gamma_nt^{p-1}\mu(t)^\frac{1}{k}\,dt\leq \int_0^\tau (-\mu'(t))t^{p-1}\mu(t)^{\frac{1}{k}-\left(1-\frac{1}{n}-\frac{1}{p} \right)\frac{p}{p-1}}\,dt+\frac{\abs{\Omega_0}^{\frac{1}{k}-\left(1-\frac{1}{n}-\frac{1}{p} \right)\frac{p}{p-1}}\abs{\Omega}}{p\beta^\frac{p}{p-1}},
    \end{equation*}
    with equality sign if we consider $\phi$ instead of $\mu$ (we remark that, from the choice of $k$ and $p$, the exponent $\frac{1}{k}-(1-\frac{1}{n}-\frac{1}{p})\frac{p}{p-1}$ is positive). For simplicity we set
    \begin{equation*}
        \eta={\frac{1}{k}-\left(1-\frac{1}{n}-\frac{1}{p} \right)\frac{p}{p-1}},\quad C=\frac{\abs{\Omega_0}^\eta\abs{\Omega}}{p\beta^\frac{p}{p-1}}.
    \end{equation*}
    Since $\mu$ is nonincreasing, the last inequality gives us
    \begin{equation}\label{24}
        \int_0^\tau \gamma_n t^{p-1}\mu(t)^\frac{1}{k}\,dt\leq\int_0^\tau -t^{p-1}\mu(t)^\eta\,d\mu(t)+C.
    \end{equation}
    Setting
    \begin{equation*}
        G(l)=\int_0^l w^\eta\,dw=\frac{l^{\eta+1}}{\eta+1},
    \end{equation*}
    and integrating by parts both sides in the last inequality we obtain 
    \begin{equation*}
    \begin{split}
        \gamma_n\tau^{p-1}\int_0^\tau \mu(t)\,dt&+\tau^{p-1}G(\mu(\tau))\\
        &\leq(p-1)\left[\int_0^\tau\gamma_nt^{p-2}\int_0^t\mu(s)^\frac{1}{s}\,ds\,dt+\int_0^\tau t^{p-2}G(\mu(t))\,dt \right]+C.
    \end{split}
    \end{equation*}
    Then, if we set
    \begin{equation*}
        \xi_\mu(\tau)=\int_0^\tau\gamma_nt^{p-2}\int_0^t\mu(s)^\frac{1}{s}\,ds\,dt+\int_0^\tau t^{p-2}G(\mu(t))\,dt, 
    \end{equation*}
    we are under the hypothesis of Gronwall's Lemma and then, choosing $\tau_0=v_m$, we get
    \begin{equation*}
        \gamma_n\tau^{p-2}\int_0^\tau\mu(t)^\frac{1}{k}\,dt+\tau^{p-2}G(\mu(\tau))\leq \left(\frac{(p-1)\xi_\mu(v_m)+C}{v_m}\right)\left(\frac{\tau}{v_m}\right)^{p-2}.
    \end{equation*}
    We remark that the last inequality reads as an equality if we replace $\mu$ with $\phi$. As in the proof of Theorem \ref{mainteor} we have that $\xi_\mu(v_m)\leq\xi_\phi(v_m)$ and then from the definition of $\xi_\mu,\,\xi_\phi$ we find
    \begin{equation*}
        \tau^{p-2}\left(\gamma_n\int_0^\tau \mu(t)^\frac{1}{k}\,dt+G(\mu(\tau))\right)\leq\tau^{p-2}\left(\gamma_n\int_0^\tau \phi(t)^\frac{1}{k}\,dt+G(\phi(\tau))\right),
    \end{equation*}
    that is, as $\tau\rightarrow+\infty$, and dividing by $\tau^{p-2}$
    \begin{equation*}
        \int_0^{+\infty}\mu(t)^\frac{1}{k}\,dt\leq\int_0^{+\infty}\phi(t)^\frac{1}{k}\,dt.
    \end{equation*}
    Then the comparison in the norm $\|\cdot\|_{L^{k,1}}$ is proved.\\
    Now we have to show that
    \begin{equation*}
        \|u\|_{L^{pk,p}(\Omega_0)}\leq \|v\|_{L^{pk,p}(\Omega_0^\sharp)}, \quad \forall\,0<k\leq\frac{n(p-1)}{n(p-1)-p}.
    \end{equation*}
    It is enough to prove that
    \begin{equation*}
        \int_0^{+\infty}t^{p-1}\mu(t)^\frac{1}{k}\,dt\leq\int_0^{+\infty}t^{p-1}\phi(t)^\frac{1}{k}\,dt.
    \end{equation*}
    Passing to the limit as $\tau\rightarrow+\infty$ and integrating by parts the first term on the right-hand side in \eqref{24} we have
    \begin{equation*}
        \int_0^{+\infty}\gamma_n t^{p-1}\mu(t)^\frac{1}{k}\,dt\leq (p-1)\int_0^{+\infty}t^{p-2}G(\mu(t))\,dt+C,
    \end{equation*}
    where the last inequality holds as an equality with $\phi$ instead of $\mu$. Hence, we need to show that
    \begin{equation*}
        \int_0^{+\infty}t^{p-2}G(\mu(t))\,dt\leq\int_0^{+\infty}t^{p-2}G(\phi(t))\,dt.
    \end{equation*}
    Multiplying \eqref{20} by $t^{p-1}G(\mu(t))\mu(t)^{-\left(1-\frac{1}{n}-\frac{1}{p}\right)\frac{p}{p-1}}$ and integrating on $[0,\tau]$, $\tau\geq v_m$, we find
    \begin{equation}\label{777}
        \begin{split}
            \int_0^\tau \gamma_n t^{p-1}G(\mu(t))\,dt&\leq\int_0^\tau -\mu'(t)t^{p-1}G(\mu(t))\mu(t)^{-\left(1-\frac{1}{n}-\frac{1}{p}\right)\frac{p}{p-1}}\,dt\\
            &+\frac{1}{\beta^\frac{1}{p-1}}\int_0^\tau t^{p-1}G(\mu(t))\mu(t)^{-\left(1-\frac{1}{n}-\frac{1}{p}\right)\frac{p}{p-1}}\int_{\partial U_t^{ext}}\frac{1}{\tilde u}\,\dH\,dt.
        \end{split}
    \end{equation}
    From Lemma \ref{LemmaPositivis} we know that $G(l)l^{-\left(1-\frac{1}{n}-\frac{1}{p}\right)\frac{p}{p-1}}$ is increasing, then when we compose it with the decreasing function $\mu(\cdot)$ we find a decreasing function, this gives us the following
    \begin{equation*}
        G(\mu(t))\mu(t)^{-\left(1-\frac{1}{n}-\frac{1}{p}\right)\frac{p}{p-1}}\leq G(|\Omega_0|)|\Omega_0|^{-\left(1-\frac{1}{n}-\frac{1}{p}\right)\frac{p}{p-1}} \quad\forall\,t\in[0,\tau]
    \end{equation*}
    and we can apply Lemma \ref{lemma3.5} in \eqref{777} finding
    \begin{equation*}
        \int_0^\tau \gamma_n t^{p-1}G(\mu(t))\,dt\leq\int_0^\tau -\mu'(t)t^{p-1}G(\mu(t))\mu(t)^{-\left(1-\frac{1}{n}-\frac{1}{p}\right)\frac{p}{p-1}}\,dt+C_2,
    \end{equation*}
    where
    \begin{equation*}
        C_2=\frac{G(\abs{\Omega_0})\abs{\Omega_0}^{\frac{-n(p-1)+p}{n(p-1)}}\abs{\Omega}}{p\beta^\frac{p}{p-1}}, 
    \end{equation*}
    and the previous inequality reads as equality if we replace $\mu$ with $\phi$. Then, integrating by parts and rearranging the terms, we find
    \begin{equation*}
        \begin{split}
            \tau\int_0^\tau \gamma_n t^{p-2}G(\mu(t))\,dt&+\tau\int_0^\tau t^{p-2}\frac{d}{dt}G_2(\mu(t))\,dt \\
            &\leq\int_0^\tau\int_0^t \gamma_n r^{p-2}G(\mu(r))\,dr\,dt+\int_0^\tau \int_0^t r^{p-2}\frac{d}{dr}G_2(\mu(r))\,dr\,dt,
        \end{split}
    \end{equation*}
    where
    \begin{equation*}
        G_2(l)=\int_0^l G(w)w^{-\left(1-\frac{1}{n}-\frac{1}{p}\right)\frac{p}{p-1}}\,dw.
    \end{equation*}
    Then applying Lemma \ref{gronwall lemma} with $q=2$ and $\tau_0=v_m$, we get
    \begin{equation}\label{778}
        \begin{split}
            \int_0^\tau \gamma_n t^{p-2}G(\mu(t))\,dt&+\int_0^\tau t^{p-2}\frac{d}{dt}G_2(\mu(t))\,dt \\
            &\leq \frac{1}{v_m}\left\{\int_0^{v_m}\int_0^t \gamma_n r^{p-2}G(\mu(r))\,dr\,dt+\int_0^{v_m}\int_0^t r^{p-2}\frac{d}{dr}G_2(\mu(r))\,dr\,dt+C_2 \right\}.
        \end{split}
    \end{equation}
    We observe that the previous inequality holds with equality sign if $\mu$ is replaced by $\phi$. Then, since $G(\cdot)$ is an increasing function we have $G(\mu(r))\leq G(\phi(r))$ for all $r\in[0,v_m]$, since $\mu(r)\leq\phi(r)$ for such $r$. On the other hand, we have that $G_2(\mu(\cdot))$ is decreasing, then, since $\phi(r)\equiv\abs{\Omega_0}$ for all $r\leq v_m$, we find
    \begin{equation*}
        \frac{d}{dr}G_2(\mu(r))<0=\frac{d}{dr}G_2(\phi(r))\quad\forall\, r\leq v_m.
    \end{equation*}
    Then from \eqref{778} we get
    \begin{equation*}
        \int_0^\tau \gamma_n t^{p-2}G(\mu(t))\,dt+\int_0^\tau t^{p-2}\frac{d}{dt}G_2(\mu(t))\,dt\leq \int_0^\tau \gamma_n t^{p-2}G(\phi(t))\,dt+\int_0^\tau t^{p-2}\frac{d}{dt}G_2(\phi(t))\,dt,
    \end{equation*}
    that is
    \begin{equation*}
        \begin{split}
            \int_0^\tau\gamma_n t^{p-2}G(\phi(t))\,dt&-\int_0^\tau\gamma_n t^{p-2}G(\mu(t))\,dt\\
            &\geq \tau^{p-2}[G_2(\mu(\tau))-G_2(\phi(\tau))]+(p-2)\int_0^\tau t^{p-3}[G_2(\phi(t))-G_2(\mu(t))]\,dt.
        \end{split}
    \end{equation*}
    If $p=2$, then sending $\tau\rightarrow+\infty$ the claim is proved (proceeding as in \eqref{vanish}). Otherwise we have to prove
    \begin{equation*}
        \int_0^\tau t^{p-3}G_2(\phi(t))\,dt\geq \int_0^\tau t^{p-3}G_2(\mu(t))\,dt.
    \end{equation*}
    We proceed as in the proof of Theorem \ref{mainteor}. Define
    \begin{equation*}
        G_j(l)=\int_0^l G_{j-1}(w)w^{-\frac{n(p-1)-p}{n(p-1)}}, \quad j=3,\dots,p.
    \end{equation*}
    Then, multiplying in \eqref{20} by $t^{p-1}G_{m-1}(\mu(t))\mu(t)^{-\left(1-\frac{1}{n}-\frac{1}{p}\right)\frac{p}{p-1}}$ and integrating on $[0,\tau]$, $\tau\geq v_m$, we find
    \begin{equation*}
        \begin{split}
            \int_0^\tau \gamma_nt^{p-1}G_{m-1}(\mu(t))\,dt&\leq \int_0^\tau -\mu'(t)t^{p-1}G_{m-1}(\mu(t))\mu(t)^{-\left(1-\frac{1}{n}-\frac{1}{p}\right)\frac{p}{p-1}}\,dt\\
            &+\frac{1}{\beta^\frac{1}{p-1}}\int_0^\tau t^{p-1}G_{m-1}(\mu(t))\mu(t)^{-\left(1-\frac{1}{n}-\frac{1}{p}\right)\frac{p}{p-1}}\int_{\partial U_t^{ext}}\frac{1}{\tilde u}\,\dH\,dt.
        \end{split}
    \end{equation*}
    Then, since $G_{m-1}(l)l^{-\left(1-\frac{1}{n}-\frac{1}{p}\right)\frac{p}{p-1}}$ is increasing, applying Lemma \ref{lemma3.5}, we get
    \begin{equation*}
        \int_0^\tau \gamma_nt^{p-1}G_{m-1}(\mu(t))\,dt\leq\int_0^\tau -\mu'(t)t^{p-1}G_{m-1}(\mu(t))\mu(t)^{-\left(1-\frac{1}{n}-\frac{1}{p}\right)\frac{p}{p-1}}\,dt+C_{m-1},
    \end{equation*}
    where 
    \begin{equation*}
        C_{m-1}=\frac{G_{m-1}(\abs{\Omega_0})\abs{\Omega_0}^{\frac{-n(p-1)+p}{n(p-1)}}\abs{\Omega}}{p\beta^\frac{p}{p-1}}.
    \end{equation*}
    Then, integrating by parts and rearranging the terms, we obtain
    \begin{equation*}
        \begin{split}
            &\tau^{m-1}\int_0^\tau \gamma_nt^{p-m}G_{m-1}(\mu(t))\,dt+\tau^{m-1}\int_0^\tau t^{p-m}\frac{d}{dt}G_m(\mu(t))\,dt\\
            &\leq (m-1)\left\{\int_0^\tau t^{m-2}\int_0^t\gamma_n r^{p-m}G_{m-1}(\mu(r))\,dr\,dt+\int_0^\tau t^{m-2}\int_0^t r^{p-m}\frac{d}{dr}G_m(\mu(r))\,dr\,dt \right\}+C_{m-1}.
        \end{split}
    \end{equation*}
    Then, applying Lemma \ref{gronwall lemma} with $q=m$, we find
    \begin{equation*}
        \begin{split}
            &\int_0^\tau\gamma_nt^{p-m}G_{m-1}(\mu(t))\,dt+\int_0^\tau t^{p-m}\frac{d}{dt}G_m(\mu(t))\,dt\\
            &\leq\frac{m-1}{v_m}\left[\int_0^{v_m} t^{m-2}\int_0^t\gamma_n r^{p-m}G_{m-1}(\mu(r))\,dr\,dt+\int_0^{v_m} t^{m-2}\int_0^t r^{p-m}\frac{d}{dr}G_m(\mu(r))\,dr\,dt\right]\left(\frac{\tau}{v_m}\right)^{m-2}\\&+\frac{C_{m-1}}{v_m}\left(\frac{\tau}{v_m}\right)^{m-2},
        \end{split}
    \end{equation*}
    and, again, the previous inequality reads as an equality with $\phi$ instead of $\mu$. Then, since $G_{m-1}(\mu(r))\leq G_{m-1}(\phi(r))$ for all $r\leq v_m$ and in this range $\frac{dG_m(\mu(r))}{dr}<0=\frac{dG_m(\phi(r))}{dr}$, we find
    \begin{equation*}
        \begin{split}
            \int_0^\tau\gamma_nt^{p-m}G_{m-1}(\mu(t))\,dt&+\int_0^\tau t^{p-m}\frac{d}{dt}G_m(\mu(t))\,dt\\
            &\leq\int_0^\tau\gamma_nt^{p-m}G_{m-1}(\phi(t))\,dt+\int_0^\tau t^{p-m}\frac{d}{dt}G_m(\phi(t))\,dt, 
        \end{split}
    \end{equation*}
    that is
    \begin{equation*}
        \begin{split}
            \int_0^\tau \gamma_nt^{p-m}G_{m-1}(\phi(t))\,dt&-\int_0^\tau \gamma_nt^{p-m}G_{m-1}(\mu(t))\,dt\\
            &\geq \tau^{p-m}[G_m(\mu(\tau))-G_m(\phi(\tau))]+(p-m)\int_0^\tau t^{p-m-1}[G_m(\phi(t))-G_m(\mu(t))]\,dt.
        \end{split}
    \end{equation*}
    For $m=p$, we can conclude sending $\tau\rightarrow+\infty$ (proceeding as in \eqref{vanish}).
\end{proof}

\section{Optimization of the p-Torsion and Robin eigenvalue}\label{section5}
The results proved in \S\ref{section4} allow us to address some optimization problems. In this section, we will prove that, among all the multiply connected sets with prescribed measures of the exterior domain and the holes, the annulus maximizes the p-torsional rigidity, for all $p>1$, and minimizes the Robin eigenvalue, where $p\geq n$.
We define the p-torsion functional as
\begin{equation*}
    T_p(\Omega,\beta)=\max\left\{\frac{\left(\int_{\Omega_0}\varphi\,dx\right)^p}{\int_{\Omega_0}\abs{\nabla\varphi}^p\,dx+\beta\int_{\partial\Omega_0}\varphi^p\,\dH}:\, \varphi \in W^{1,p}(\Omega_0), \nabla\varphi\equiv0 \text{ on }  \Omega_i \, \forall i=1, \dots m\right\}
\end{equation*}
and we recall that the maximizer of such a functional solves the following problem
\begin{equation}\label{torsprob}
    \begin{cases}
        -\Delta_p w=1 & \text{ in }\Omega\\
        \abs{\nabla w}^{p-2}\frac{\partial w}{\partial \nu}+\beta\abs{w}^{p-2}w=0 & \text{ on }\partial\Omega_0\\
        w=c_i & \text{ on }\partial\Omega_i.
    \end{cases}
\end{equation}
Let $w$ be the solution to \eqref{torsprob} and $\tilde w$ be its constant extension, then the $p$-torsional rigidity is defined analogously
\begin{equation*}
    T_p(\Omega,\beta)=\|\tilde w\|_{L^1(\Omega_0)}.
\end{equation*}
Then, as a consequence of Theorem \ref{mainteo2} we have directly the following optimization result.
\begin{corollario}
    Let $p>1$, then we have
    \begin{equation*}
        T_p(\Omega,\beta)\leq T_p(A_\Omega,\beta).
    \end{equation*}
\end{corollario}
\begin{proof}
    Applying Theorem \ref{mainteo2} to the solution of \eqref{torsprob} the conclusion follows from the comparison with $k=1$. 
\end{proof}

Likewise, the first eigenvalue of the Robin p-laplacian is defined through the following Rayleigh quotient
\begin{equation*}
    \lambda_p(\Omega,\beta)=\min\left\{\frac{\int_{\Omega_0}\abs{\nabla\varphi}^p\,dx+\beta\int_{\partial\Omega_0}\varphi^p\,\dH}{\int_{\Omega_0}\varphi^p\,dx}:\, \varphi \in W^{1,p}(\Omega_0), \nabla\varphi\equiv0 \text{ on }  \Omega_i \, \forall i=1, \dots m\right\},
\end{equation*}
or as the smallest real number such that the following problem
\begin{equation}\label{eig}
        \begin{cases}
            -\Delta_pu=\lambda_{p}(\Omega,\beta)\abs{u}^{p-2}u\quad\text{ in }\Omega \\
            \abs{\nabla u}^{p-2}\frac{\partial u}{\partial\nu}+\beta \abs{u}^{p-2}u=0\quad\text{ on }\partial\Omega_0\\
            u=c_i \quad\text{ on }\partial\Omega_i,
        \end{cases}
    \end{equation}
where $c_i$ are implicitly defined as in Remark \ref{Def Ci}, has a solution. As observed in Lemma \ref{segno ci} we have that the solution $u\geq0$, furthermore due to Harnack inequality we can infer that $u>0$.
\begin{prop}
    The first eigenvalue $\lambda_{p}(\Omega,\beta)$ is simple.
\end{prop}
\begin{proof}
    Let $u,\,v$ be two eigenfunctions related to the first eigenvalue, without loss of generality we assume $u\geq v$, and define
    \begin{equation*}
        \varphi_1=\frac{u^p-v^p}{u^{p-1}},\qquad\varphi_2=\frac{v^p-u^p}{v^{p-1}}.
    \end{equation*}
    We remark that $\varphi_1,\,\varphi_2$ are still admissible in the class $\mathcal{X}$ defined in \eqref{Spazio} and then we can use them as test in \eqref{eqdebole} obtaining
    \begin{equation*}
        \begin{split}
            &\int_{\Omega}\abs{\nabla u}^{p-2}\nabla u\cdot\nabla\varphi_1\,dx+\beta\int_{\partial\Omega_0}u^{p-1}\varphi_1\,\dH=\int_\Omega\lambda_p u^{p-1}\varphi_1\,dx+\lambda_p\sum_{i=1}^m c_i\int_{\Omega_i} \varphi_1\,dx,\\
            &\int_{\Omega}\abs{\nabla v}^{p-2}\nabla v\cdot\nabla\varphi_2\,dx+\beta\int_{\partial\Omega_0}v^{p-1}\varphi_2\,\dH=\int_\Omega\lambda_p v^{p-1}\varphi_2\,dx+\lambda_p\sum_{i=1}^m c_i\int_{\Omega_i} \varphi_2\,dx.
        \end{split}
    \end{equation*}
    We remark that the terms on $\Omega_i$ vanish by the definition of $\varphi$, then by summing the two equations
    \begin{equation}\label{IdUtile}
        \begin{split}
            0&=\int_\Omega\left\{\left[1+(p-1)\left(\frac{v}{u}\right)^p\right]\abs{\nabla u}^p+\left[1+(p-1)\left(\frac{u}{v}\right)^p\right]\abs{\nabla v}^p\right\}\,dx\\ &-p\int_\Omega\left[\left(\frac{v}{u}\right)^{p-1}\abs{\nabla u}^{p-2}+\left(\frac{u}{v}\right)^{p-1}\abs{\nabla v}^{p-2}\right]\nabla u\cdot\nabla v\,dx\\&=\int_\Omega(u^p-v^p)(\abs{\nabla\log u}^p-\abs{\nabla\log v}^p)\,dx\\&-p\int_\Omega\left[v^p\abs{\nabla\log u}^{p-2}\nabla\log u\cdot(\nabla\log v-\nabla\log u)+u^p\abs{\nabla\log v}^{p-2}\nabla\log v\cdot(\nabla\log u-\nabla\log v)\right]\,dx.
        \end{split}
    \end{equation}
    We recall the following inequality (for the proof see \cite{L}) that holds for $p\geq2$ and for all $w_1,w_2\in\R^n$
    \begin{equation}\label{DisUtile}
        \abs{w_2}^p-\abs{w_1}^p\geq p\abs{w_1}^{p-2}w_1\cdot(w_2-w_1)+\frac{\abs{w_2-w_1}^p}{2^{p-1}-1}.
    \end{equation}
    Applying twice \eqref{DisUtile}, first with $w_2=\nabla\log u$,  $w_1=\nabla\log v$, second with $w_1=\nabla\log u$,  $w_2=\nabla\log v$, and multiplying respectively by $u^p$ and $v^p$ we obtain
    \begin{equation*}
        \begin{split}
            u^p\left(\abs{\nabla \log u}^p-\abs{\nabla\log v}^p \right)&\geq u^pp\abs{\nabla\log v}^{p-2}\nabla\log v\cdot (\nabla\log u-\nabla\log v) \\
            &+\frac{u^p}{2^{p-1}-1}\abs{\nabla\log u-\nabla\log v}^p,\\
            v^p\left(\abs{\nabla \log v}^p-\abs{\nabla\log u}^p \right)&\geq v^pp\abs{\nabla\log u}^{p-2}\nabla\log u\cdot (\nabla\log v-\nabla\log u) \\
            &+\frac{v^p}{2^{p-1}-1}\abs{\nabla\log v-\nabla\log u}^p.
        \end{split}
    \end{equation*}
    
    Then, summing the two previous inequalities and combining it with \eqref{IdUtile}, we find
    \begin{equation*}
        \begin{split}
            0&=\int_\Omega(u^p-v^p)\left(\abs{\nabla \log u}^p-\abs{\nabla\log v}^p \right)\,dx-p\int_\Omega v^p\abs{\nabla\log u}^{p-2}\nabla\log u\cdot(\nabla\log v-\nabla\log u)\,dx\\
            &-p\int_\Omega u^p\abs{\nabla\log v}^{p-2}\nabla\log v\cdot(\nabla\log u-\nabla\log v)\,dx\\
            &\geq\frac{1}{2^{p-1}-1}\int_\Omega (u^p+v^p)\abs{\nabla\log u-\nabla\log v}^p\,dx\geq0,
        \end{split}
    \end{equation*}
    that implies
    \begin{equation*}
        \frac{1}{2^{p-1}-1}\int_\Omega\left(\frac{1}{v^p}+\frac{1}{u^p}\right)\abs{v\nabla u-u\nabla v}^p\,dx=0.
    \end{equation*}
    Hence, we obtain that $v\nabla u=u\nabla v$ a.e., therefore there exists a constant $K$ for which $u=Kv$.
\end{proof} 

\begin{teorema}
    Let $p\geq n$ be an integer, then we have
    \begin{equation*}
        \lambda_{p}(\Omega,\beta)\geq\lambda_{p}(A_\Omega,\beta).
    \end{equation*}
\end{teorema}
\begin{proof}
    Let us consider the problem \eqref{eig} and the symmetrized one in the spirit of Theorem \ref{mainteor}
    \begin{equation}\label{eigsymm}
        \begin{cases}
            -\Delta_pz=\lambda_{p}(\Omega,\beta)\abs{u^\sharp}^{p-2}u^\sharp\quad\text{ in }A_\Omega \\
            \abs{\nabla z}^{p-2}\frac{\partial z}{\partial\nu}+\beta\abs{z}^{p-2} z=0\quad\text{ on }\partial\Omega_0^\sharp\\
            z=\overline c \quad\text{ on }\partial S^\sharp.
        \end{cases}
    \end{equation}
    If we multiply by $z$ in \eqref{eigsymm}, integrating on $A_\Omega$, we find
    \begin{equation}\label{NewEqa}
        \int_{A_\Omega}\abs{\nabla z}^p\,dx+\beta\int_{\partial \Omega_0^\sharp}z^p\,\dH=-\int_{A_\Omega}z\Delta_p z\,dx=\lambda_p(\Omega,\beta)\int_{A_\Omega}\abs{u^\sharp}^{p-2}u^\sharp z\,dx.
    \end{equation}
    Then by property of symmetrization and Corollary \ref{BM}
    \[\int_{\Omega_0}\abs{\tilde u}^p\,dx=\int_{\Omega_0^\sharp}\abs{\tilde u^\sharp}^p\,dx\leq\int_{\Omega_0^\sharp}\abs{\tilde z}^p\,dx,\]
    that implies, by H\"older inequality
    \begin{equation}\label{stimarhs}
        \int_{A_\Omega}\abs{ u^\sharp}^{p-2} u^\sharp z\,dx\leq\int_{\Omega^\sharp_0}\abs{\tilde u^\sharp}^{p-2}\tilde u^\sharp \tilde z\,dx\leq\left(\int_{\Omega^\sharp_0} 
        \tilde z^p\,dx\right)^\frac{1}{p}\left(\int_{\Omega^\sharp_0}\abs{\tilde u^\sharp}^p\right)^{1-\frac{1}{p}}\leq\int_{\Omega^\sharp_0} 
        \tilde z^p\,dx
    \end{equation}
    Therefore, combining \eqref{NewEqa} and \eqref{stimarhs}, we can conclude as follows
    \begin{equation*}
            \lambda_{p}(\Omega,\beta)=\frac{\int_{\Omega_0^\sharp}\abs{\nabla z}^p\,dx+\beta\int_{\partial\Omega_0^\sharp}z^p\,\dH}{\int_{A_\Omega}\abs{ u^\sharp}^{p-2} u^\sharp z\,dx}\geq\frac{\int_{\Omega_0^\sharp}\abs{\nabla z}^p\,dx+\beta\int_{\partial\Omega_0^\sharp}z^p\,\dH}{\int_{\Omega_0^\sharp} z^p\,dx}\geq\lambda_{p}(A_\Omega,\beta).
    \end{equation*}
\end{proof}

\subsubsection*{Acknowledgements}
We would like to thank Prof. Cristina Trombetti for the valuable advice that helped us to achieve these results.
\subsubsection*{Declarations}

\paragraph{Funding}
The authors were partially supported by Gruppo Nazionale per l’Analisi Matematica, la Probabilità e le loro Applicazioni
(GNAMPA) of Istituto Nazionale di Alta Matematica (INdAM).   

\paragraph{Data Availability} All data generated or analysed during this study are included in this published article.
	
\paragraph{Competing Interests} We declare that we have no financial and personal relationships with other people or organizations.

\bibliographystyle{plain}
\bibliography{biblio}

\Addresses

\end{document}